\documentclass[a4paper, 11pt]{article}
\usepackage{float}
\usepackage[T2A]{fontenc}
\usepackage[utf8x]{inputenc}
\usepackage[english]{babel}

\usepackage{latexsym}
\usepackage{geometry}
\usepackage{graphicx}
\usepackage{amsfonts}
\usepackage{amsmath}
\usepackage{amsthm}
\usepackage{amssymb}
\usepackage{enumerate}
\usepackage{mathtools}
\usepackage[noadjust]{cite}
\usepackage{centernot}
\usepackage{indentfirst}
\usepackage{xcolor}

\newtheorem{theorem}{Theorem}[section]
\newtheorem{corollary}[theorem]{Corollary}
\newtheorem{proposition}[theorem]{Proposition}
\newtheorem{lemma}[theorem]{Lemma}
\theoremstyle{definition}
\newtheorem{definition}[theorem]{Definition}

\newtheorem{notation}[theorem]{Notation}
\newtheorem{example}[theorem]{Example}
\newtheorem{remark}[theorem]{Remark}

\numberwithin{equation}{section}
\allowdisplaybreaks

\def\A{\mathcal{A}}

\def\R{\mathbb{R}}
\def\C{\mathbb{C}}
\def\F{\mathbb{F}}
\def\Okubo{\mathcal{O}}

\DeclareMathOperator{\n}{n}
\DeclareMathOperator{\s}{s}
\DeclareMathOperator{\tr}{tr}
\DeclareMathOperator{\spn}{span}

\DeclareMathOperator{\D}{d}
\DeclareMathOperator{\diam}{diam}

\DeclareMathOperator{\algop}{alg}
\DeclareMathOperator{\chrs}{char}

\newcommand\scpr[2]{\n ( #1, #2 )}
\newcommand\liepr[2]{[ #1, #2 ]}
\newcommand\alg[1]{\algop \langle #1 \rangle}

\newcommand{\frsl}{{\mathfrak{sl}}}

\providecommand{\keywords}[1]{\textbf{Keywords:} #1}
\providecommand{\msc}[1]{\textbf{MSC 2020:} #1}

\newcommand\freefootnote[1]{%
    \bgroup
    \renewcommand\thefootnote{\fnsymbol{footnote}}%
    \renewcommand\thempfootnote{\fnsymbol{mpfootnote}}%
    \footnotetext[0]{#1}%
    \egroup
}

\begin{document}

\title{Commuting graphs of Okubo algebras}
\author{
Svetlana Zhilina$^{a,b}$ and Danil Pavlinov$^{a,b}$
}
\date{\small \em
$^a$Department of Mathematics and Mechanics, Lomonosov Moscow State\\ University, Moscow, 119991, Russia\\
$^b$Moscow Center of Fundamental and Applied Mathematics, Moscow, 119991, Russia
}

\maketitle

\begin{abstract}
Commuting graphs of Okubo algebras are considered, and the problem of their connectivity is studied. The commuting graph of a pseudo-octonion algebra $P_8(\F)$ over a field~$\F$, $\chrs \F \neq 3$, that contains a primitive cubic root of unity is shown to be isomorphic to the commuting graph of the matrix algebra $M_3(\F)$. As a consequence, if the field~$\F$ is algebraically closed, then the diameter of the commuting graph for the unique Okubo algebra over~$\F$ equals~$4$. It is shown that the commuting graph of the real division Okubo algebra is connected, and its diameter also equals~$4$. The proof of this result relies on the fact that, given any two idempotents in an arbitrary Okubo algebra, the intersection of their centralizers is always nonzero.
\end{abstract}


\msc{05C25, 17A75.}

\freefootnote{The authors' research was supported by the Ministry of Science and Higher Education of the Russian Federation as part of the program of the Moscow Center of Fundamental and Applied Mathematics under agreement No. 075-15-2025-345.}

\freefootnote{Email addresses: \texttt{s.a.zhilina@gmail.com} (Svetlana Zhilina), \texttt{pavlinov.d.aa@gmail.com} (Danil Pavlinov)}

\section{Introduction}

The study of binary relations on algebraic structures leads naturally to consideration of the corresponding relation graphs. In the case of algebras over a field, the most widely studied graphs by now are commuting graphs, orthogonality graphs, and zero divisor graphs. The current paper is devoted to the commuting graphs of an important class of non-associative algebras, namely, Okubo algebras.

The definition of the commuting graph of an arbitrary algebra $\A$ over a field~$\F$ (possibly noncommutative or non-associative) is closely related to the definition of the commutative center. The {\em commutative center} of an algebra $\A$ is defined as the set of those elements in~$\A$ which commute with all elements:
\[
    C_\A = \big\{ c \in \A \: | \: cb=bc \quad \forall \, b \in \A \big\}.
\]

\begin{notation}
For any subset $X$ of a linear space $V$ over~$\F$ we denote the set of lines passing through nonzero elements of~$X$ by
\[
    \mathbb{P}(X) = \{ \mathbb{F} x \; | \; x \in X \setminus \{ 0 \} \}.
\]
\end{notation}
	
\begin{definition} \label{definition:graphs}
The {\em commuting graph} $\Gamma_C(\A)$ of an algebra $\A$ over a field~$\F$ is the graph whose vertices correspond to the elements of the set
\[
    \mathbb{P}(\A/C_\A) = \{ [a] = \mathbb{F}a + C_{\A} \; | \; a \in \A \setminus C_\A \},
\]
and two distinct vertices $[a]$ and $[b]$ are adjacent if and only if $ab = ba$.
\end{definition}

Henceforth, when speaking about the vertices of the graph defined above, we will not distinguish between a nonzero element $a$ and the corresponding vertex $[a] = \F a + C_{\A}$. We will use the notation $a \leftrightarrow b$ to indicate that the vertices $a$ and $b$ commute (i.e., either coincide or are adjacent). Recall that, in a graph $\Gamma$, $\D(a,b)$ denotes the {\em distance} between two vertices $a$ and $b$, and $\diam(\Gamma) = \sup\limits_{a,b \in \Gamma} \D(a,b)$ is the {\em diameter} of $\Gamma$.

The aim of the current paper is to study the commuting graphs of Okubo algebras, namely, to resolve the question of their connectivity and to compute their diameters. We will particularly focus on the case of the pseudo-octonion algebra \(P_8(\F)\) over a field~$\F$, $\chrs \F \neq 3$, containing a primitive cube root of unity, and also on the real division Okubo algebra --- the so-called real pseudo-octonions. In particular, if the field~$\F$ is algebraically closed and $\chrs \F \neq 3$, then the pseudo-octonion algebra is the unique Okubo algebra over~$\F$ up to isomorphism.

The main distinguishing feature of the above cases is a close connection with matrix algebras, whose commuting graphs have been actively studied by various authors~\cite{Shitov, Akbari, Akbari2, Akbari3, DKO12, Shitov16}. At present, the main results on the diameter of $\Gamma_C(M_n(\F))$ --- the commuting graph of the algebra of $n \times n$ matrices over~$\F$ --- are as follows. Over any field~$\F$, the graph $\Gamma_C(M_2(\F))$ is disconnected~\cite[Remark~8]{Akbari3}. By~\cite[Theorems~3 and~17]{Akbari}, if for some $n \geq 3$ the graph $\Gamma_C(M_n(\F))$ is connected, then $4 \leq \diam \Gamma_C(M_n(\F)) \leq 6$. If the field~$\F$ is algebraically closed and $n \geq 3$, then $\Gamma_C(M_n(\F))$ is connected, and its diameter equals~$4$, see \cite[Corollary~7]{Akbari}. A similar statement holds in the case when $\F = \R$ \cite{Shitov}. The first example of an algebra $M_n(\F)$ such that the largest possible diameter of the commuting graph is achieved was given in~\cite{Shitov16}. In~\cite{Morrison}, the graphs $\Gamma_C(M_n(\mathbb{Q}_p))$ were considered, where $\mathbb{Q}_p$ denotes the field of $p$-adic numbers, and it was shown that, for a prime $q \geq 7$, the diameter of $\Gamma_C(M_{2q}(\mathbb{Q}_2))$ equals~$6$. This provides the first example of a commuting graph with the largest possible diameter that does not use the axiom of choice. Still, it cannot be guaranteed that the graph $\Gamma_C(M_n(\F))$ is connected in general; for example, for every $n \geq 3$ the graph $\Gamma_C(M_n(\mathbb{Q}))$ of the matrix algebra over the rational numbers is disconnected, see~\cite[Remark~8]{Akbari2}. This result follows from Theorem~6 of~\cite{Akbari2}, which states that the graph $\Gamma_C(M_n(\F))$ is connected if and only if $n \geq 3$ and every field extension of~$\F$ of degree~$n$ contains a proper intermediate field.

The structure of this paper is as follows: In Section~2 we present the construction of Okubo algebras over fields of characteristic not three via pseudo-octonion algebras. This section also contains the main statements about Okubo algebras, their commutative center, and nonzero idempotents that are used throughout the text.

In Section~3 we consider the commuting graph of the pseudo-octonion algebra $P_8(\F)$ over a field~$\F$, $\chrs \F \neq 3$, which contains a primitive cube root of unity. By Proposition~\ref{proposition:commuting-graph-isomorphism}, the graph $\Gamma_C(P_8(\F))$ is isomorphic to the commuting graph $\Gamma_C(M_3(\F))$ of the algebra of $3 \times 3$ matrices over the field~$\F$. It follows that the diameter of the commuting graph of the unique Okubo algebra over an algebraically closed field~$\F$, $\chrs \F \neq 3$, equals~$4$. However, if the field~$\F$ is not algebraically closed, then the commuting graph $\Gamma_C(P_8(\F))$ can be disconnected, see Corollary~\ref{corollary:disconnected}.

In Section~4 we study paths between nonzero idempotents in the commuting graph of an Okubo algebra~$\Okubo$ over an arbitrary field~$\F$. By Corollary~\ref{corollary:idempotents-distance-two}, the distance between any two idempotents in $\Gamma_C(\Okubo)$ is at most~$2$, i.e., the intersection of their centralizers is nonzero.

Section~5 is devoted to the study of the commuting graph of the real pseudo-octonion algebra $\widetilde{P}_8(\R)$. By Lemma~\ref{lemma:path-to-idempotent}, in $\Gamma_C(\widetilde{P}_8(\R))$ any nonzero element is adjacent to an idempotent. Combining this statement with the results of Section~4, we obtain Theorem~\ref{theorem:real-commutativity-graph}, which states that the commuting graph $\Gamma_C(\widetilde{P}_8(\R))$ is connected, and its diameter equals~$4$, as in the case of an Okubo algebra over an algebraically closed field.
	
\section{Okubo algebras and main properties}

We use the definition of an Okubo algebra given in \cite[Section~1]{Eld99}. Assume first that the characteristic of the field~$\F$ is distinct from~$3$, and~$\F$ contains a primitive cube root of unity $\omega$. In particular, the second condition is satisfied if~$\F$ is algebraically closed. We set 
\[
\mu = \frac{1-\omega^2}{3}.
\]
Let $\frsl_3(\F)$ denote the Lie algebra of traceless $3 \times 3$ matrices over~$\F$. Define a new non-associative product ``\(*\)'' on $\frsl_3(\F)$ by
\begin{equation} \label{equation:product}
x * y = \mu x y + (1-\mu) y x - \frac{\tr(xy)}{3} I,
\end{equation}
where $I \in M_3(\F)$ is the identity matrix. The resulting algebra, denoted by $P_8(\F)$, is called the {\em pseudo-octonion algebra} over~$\F$.

Consider then the quadratic form
\begin{equation} \label{equation:norm-definition}
\n(x) = \frac{1}{6}\tr(x^2)
\end{equation}
on the algebra $P_8(\F)$. As noted in \cite[p.~101]{Eld99}, this definition makes sense even when~$\chrs \F = 2$. Indeed, by \cite[p.~1027, Lemma]{Fau88}, for any matrix $x \in M_3(\F)$ the identity
\[
\tr(x^2) = (\tr x)^2 - 2\,\s(x)
\]
holds, where $\s(x)$ is the quadratic form appearing as the coefficient of the characteristic polynomial
\begin{equation} \label{equation:ch}
\det(\lambda I - x) = \lambda^3 - \tr(x)\lambda^2 + \s(x)\lambda - \det(x).
\end{equation}
Thus $\tr(x^2)$ can be ``divided by 2'' for any $x \in \frsl_3(\F)$, and
\begin{equation} \label{equation:norm-division-by-3}
\n(x) = -\frac{1}{3}\s(x).
\end{equation}

\begin{definition}
Let $(\A, +, *)$ be an arbitrary (possibly non-unital and non-associative) algebra over a field~$\F$. Assume that $\A$ is equipped with a strictly nondegenerate quadratic form $\n(\cdot)$, i.e., the associated symmetric bilinear form $\scpr{a}{b} = \n(a+b) - \n(a) - \n(b)$ is nondegenerate on $\A$. Then $\A$ is called a {\em composition algebra} if the quadratic form $\n(\cdot)$ admits composition, i.e., $\n(a*b) = \n(a)\n(b)$ for all $a,b \in \A$. Here $\n(\cdot)$ is called the {\em norm}.
\end{definition}

\begin{definition} \label{definition:symmetric-composition}
A composition algebra $(\A,*,\n)$ is called {\em symmetric} if
\[ 
\scpr{x*y}{z} = \scpr{x}{y*z}
\]
for all $x,y,z \in \A$. 
\end{definition}
By \cite[Lemma~1.1]{EM93}, the algebra $P_8(\F)$ is a symmetric composition algebra.

\begin{theorem}[{\cite[Theorem~1]{Oku78D}}, {\cite[(34.1)]{KMRT98}}] \label{theorem:symmetric-flexible}
Let $(\A,*,\n)$ be a composition algebra. The following conditions are equivalent:
\begin{enumerate}[\rm (1)]
\item $\A$ is symmetric;
\item $(x*y)*x = x*(y*x) = \n(x)\,y$ for all $x,y \in \A$.
\end{enumerate}
\end{theorem}

Any algebra $\A$ satisfying condition (2) of the theorem above is a {\em flexible algebra}. Namely, for all $x,y \in \A$ the flexibility identity holds:
\[
x*(y*x) = (x*y)*x.
\]

By linearizing the equality in Theorem~\ref{theorem:symmetric-flexible}(2), we obtain
\begin{equation} \label{equation:lin-of-symm}
(x*y)*z + (z*y)*x = x*(y*z) + z*(y*x) = \n(x,z)y.
\end{equation}

Now assume that the field~$\F$ is not necessarily algebraically closed, and let $\overline{\F}$ be its algebraic closure. An algebra $\A$ over~$\F$, $\chrs \F \neq 3$, is called an {\em Okubo algebra} if it is an~$\F$-form of $P_8(\overline{\F})$, i.e., 
\(
\A \otimes_{\F} \overline{\F} \;\cong\; P_8(\overline{\F}).
\)

\begin{example}
The first Okubo algebras to be constructed \cite{Oku78} were the complex algebra $P_8(\C)$ and its real form $\widetilde{P}_8(\R)$, the latter being defined as the set of traceless Hermitian matrices:
\begin{equation}
\widetilde{P}_8(\R) = \big\{ x \in M_3(\C) \: | \: x^* = x, \: \tr x = 0\big\}.
\end{equation}
The algebra $\widetilde{P}_8(\R)$ is called the {\em real pseudo-octonion algebra}.
\end{example}

The norm on $\widetilde{P}_8(\R)$ is positive definite \cite[p.~47]{Oku95}, and hence is {\em anisotropic}, i.e., if for some $a \in \widetilde{P}_8(\R)$ we have $\n(a) = 0$, then $a = 0$.

\begin{remark}
As shown in \cite[Theorem~6.6]{EM91}, there are exactly two Okubo algebras over~$\R$ up to isomorphism, namely, $\widetilde{P}_8(\R)$ and the split Okubo algebra. The norm on $\widetilde{P}_8(\R)$ is anisotropic, so it is a division algebra, whereas the norm on the split Okubo algebra is isotropic, and thus it contains zero divisors, see~\cite[Lemma~2.1]{EM91}.
\end{remark}

If $\chrs \F = 3$, then a different approach is needed to define an Okubo algebra. It was first defined in~\cite{OO81b} by presenting an explicit multiplication table, and in~\cite{EP96} a more general approach is offered, based on the ideas of Petersson's construction~\cite{Pet69}. In any case, it is universally true that any Okubo algebra over an arbitrary field~$\F$ is a symmetric composition algebra. We will use $\Okubo$ to denote an arbitrary Okubo algebra.

\begin{remark}
Consider an Okubo algebra $\Okubo$ over an arbitrary field~$\F$. Then $\Okubo_{\overline{\F}} = \Okubo \otimes_{\F} \overline{\F}$ is an Okubo algebra over $\overline{\F}$. Let $z \in C_{\Okubo}$ be an arbitrary element in the commutative center of $\Okubo$. For each $x \in \Okubo$, let $\widetilde{x} = x \otimes 1 \in \Okubo_{\overline{\F}}$. Since all elements of the form $\widetilde{x}$ commute with $\widetilde{z}$ and generate $\Okubo_{\overline{\F}}$ as a vector space over $\overline{\F}$, it follows that $\widetilde{z} \in C_{\Okubo_{\overline{\F}}}$. By~\cite[p.~5]{Eld15}, the algebra $\Okubo_{\overline{\F}}$ has a trivial commutative center, i.e., $C_{\Okubo_{\overline{\F}}} = \{0\}$. Hence we also have $C_{\Okubo} = \{0\}$ for an Okubo algebra $\Okubo$ over an arbitrary field, and the set of vertices of the commuting graph $\Gamma_C(\Okubo)$ is exactly the set of lines passing through the nonzero elements of~$\Okubo$.
\end{remark}

Note that any Okubo algebra $\Okubo$ satisfies the relation
\begin{equation} \label{equation:xxxx}
(x*x)*(x*x) = \scpr{x}{x*x}\, x - \n(x)\,x*x
\end{equation}
for all $x \in \Okubo$, see \cite[(34.3)]{KMRT98}. Together with Theorem~\ref{theorem:symmetric-flexible}(2), this implies that the subalgebra $\alg{x}$ generated by an element $x$ coincides with the linear span of $x$ and $x*x$, that is, $\spn \{x, x*x\}$.

\begin{proposition}[{\cite[(34.10)]{KMRT98}}] \label{proposition:norm-of-idempotent}
Let $e \in \Okubo$ be a nonzero idempotent. Then $\n(e) = 1$.
\end{proposition}

\begin{proof}
We have $n(e)e = e*(e*e) = e*e = e$.
\end{proof}

\begin{proposition} \label{proposition:quadratic-subalgebra}
Let $x \in \Okubo \setminus \{0\}$. Assume that one of the following three conditions holds:
\begin{enumerate}[\rm (1)]
\item $\n(x) \neq \lambda^2$ for any $\lambda \in \F$,
\item $\n(x) = \lambda^2$ for some $0 \neq \lambda \in \F$, and the set $\F x$ does not contain a nonzero idempotent,
\item $\n(x) = 0$ and $x*x \neq 0$.
\end{enumerate}
Then $\dim (\alg{x}) = 2$, i.e., $x*x \notin \F x$. Otherwise, $\dim(\alg{x}) = 1$.
\end{proposition}

\begin{proof}
Assume that $\dim(\alg{x}) = 1$ or, equivalently,
\begin{equation} \label{equation:linear-dependence}
x*x = \lambda x \quad \text{for some } \lambda \in \F.
\end{equation}
Multiplying Eq.~\eqref{equation:linear-dependence} by $x$ on the right, we obtain
\[
\n(x)x = (x*x)*x = \lambda\, x*x = \lambda^2 x,
\]
so $\n(x) = \lambda^2$.

Hence, if $\n(x)=0$, the condition \eqref{equation:linear-dependence} is equivalent to $x*x=0$. Otherwise, if $\n(x)=\lambda^2 \neq 0$, it is equivalent to the idempotency of the element $x/\lambda \in \F x$.
\end{proof}

\section{Commuting graph of the pseudo-octonion algebra}

In this section we consider the commuting graph of the pseudo-octonion algebra $P_8(\F)$ over a field~$\F$, $\operatorname{char}\F \neq 3$ and $\omega \in \F$. In particular, if the field~$\F$ is algebraically closed, then, by definition, $P_8(\F)$ is the unique Okubo algebra over~$\F$ (see also \cite[Theorem~5.1]{Eld97}, \cite[p.~4, Table~1]{Eld15} and \cite[pp.~3--4]{Eld23}). We show that the commuting graph $\Gamma_C(P_8(\F))$ is isomorphic to the graph $\Gamma_C(M_3(\F))$.

\begin{remark} \label{remark:commutativity}
Given any $a,b \in P_8(\F)$, consider the operations
\[
    \liepr{a}{b} = ab - ba, 
    \quad 
    \liepr{a}{b}_* = a*b - b*a.
\]
Then, as shown in \cite[(2.11)]{Oku78},
\(
    \liepr{a}{b}_* = (2\mu - 1)\,\liepr{a}{b}.
\)
It follows that $a$ and $b$ commute with respect to the matrix product if and only if they commute with respect to the ``$*$'' multiplication.
\end{remark}

\begin{proposition} \label{proposition:commuting-graph-isomorphism}
Let $\operatorname{char}\F \neq 3$ and $\omega \in \F$. The graphs $\Gamma_C(P_8(\F))$ and $\Gamma_C(M_3(\F))$ are isomorphic.
\end{proposition}

\begin{proof}
If $a \in M_3(\F)$, then $a - \tr(a)/3 \cdot I \in P_8(\F)$.  It is well known that the central elements of the matrix algebra $M_n(\F)$ are precisely the scalar matrices.  Moreover, adding a scalar matrix does not affect the commutativity relation.  Therefore, the map 
\begin{align*}
    \varphi: \mathbb{P}(P_8(\F)/C_{P_8(\F)}) &\rightarrow \mathbb{P}(M_3(\F)/C_{M_3(\F)}), \\
    \F a &\mapsto \F a + C_{M_3(\F)},
\end{align*}
induces an isomorphism of the commuting graphs.
\end{proof}

Akbari, Mohammadian, Radjavi, and Raja proved~\cite[Corollary 7]{Akbari} that if the field~$\mathbb{F}$ is algebraically closed and $n \geq 3$, then the commuting graph $\Gamma_C(M_n(\mathbb{F}))$ is connected, and its diameter equals~$4$.

\begin{remark} \label{remark:connection-btw-difinitions-of-commutativity-graphs}
In \cite{Akbari}, the commuting graph is defined so that its vertices are all non-central elements of the algebra $\mathcal{A}$, whereas in our case the vertices correspond to the classes of elements 
\(
    \mathbb{P}(\mathcal{A}/C_{\mathcal{A}}) = \{[a] = \F a + C_{\mathcal{A}} \mid a \in \mathcal{A} \setminus C_{\mathcal{A}} \}.
\)
Still, the above-mentioned result on the diameter of $\Gamma_C(M_n(\F))$ remains valid. Indeed,  
\begin{itemize}
    \item each path $[v_0] \leftrightarrow \dots \leftrightarrow [v_m]$ in our graph corresponds to the path $v_0 \leftrightarrow \dots \leftrightarrow v_m$ in the graph from~\cite{Akbari};
    \item any shortest path $v_0 \leftrightarrow \dots \leftrightarrow v_m$ in the graph from~\cite{Akbari} induces a path $[v_0] \leftrightarrow \dots \leftrightarrow [v_m]$ in our graph. Note that the loops are prohibited in our definition of the commuting graph. Thus, if $[v_i] = [v_{i+1}]$, then we ``glue together'' this pair of adjacent vertices, and the length of the path decreases by~$1$. However, if $m \geq 3$, then all the classes $[v_i]$ are pairwise distinct, as otherwise the original path could be shortened.
\end{itemize}
Thus, passing to elements of $\mathbb{P}(\A/C_{\A})$ as vertices of $\Gamma_C(\A)$ does not affect the diameter of the commuting graph if its value is at least~$3$.
\end{remark}

\begin{corollary} \label{corollary:commuting-graph-main}
Let~$\F$ be an algebraically closed field, $\operatorname{char}\F \neq 3$. Then the commuting graph $\Gamma_C(P_8(\F))$ is connected, and its diameter equals~$4$.
\end{corollary}

\begin{proof}
Follows immediately from Proposition~\ref{proposition:commuting-graph-isomorphism} and~\cite[Corollary 7]{Akbari}.
\end{proof}

In the case when the field~$\F$ is not algebraically closed but contains a primitive cube root of unity $\omega$, the connectivity of $\Gamma_C(P_8(\F))$ may depend on~$\F$.

\begin{theorem}[{\cite[Theorem 19]{Akbari}}] \label{theorem:prime-diameter-four}
Let~$\F$ be an arbitrary field, and $p \ge 3$ be a prime number. If the graph $\Gamma_C(M_p(\F))$ is connected, then its diameter equals~$4$.
\end{theorem}

\begin{corollary}
Let $\operatorname{char}\F \neq 3$ and $\omega \in \F$. If the graph $\Gamma_C(P_8(\F))$ is connected, then its diameter equals~$4$.
\end{corollary}

\begin{proof}
Follows immediately from Proposition~\ref{proposition:commuting-graph-isomorphism} and Theorem~\ref{theorem:prime-diameter-four}.
\end{proof}

We now show that $\Gamma_C(P_8(\F))$ can be disconnected.

\begin{theorem}[{\cite[Theorem 6]{Akbari2}}] \label{theorem:connectivity-criterion-of-graph-Mn}
Let~$\F$ be a field, and $n \geq 3$. Then the graph $\Gamma_C(M_n(\F))$ is connected if and only if every field extension of~$\F$ of degree $n$ contains a proper intermediate field.
\end{theorem}

\begin{corollary} \label{corollary:disconnected}
Let $\F = \mathbb{Q}[\omega]$ be the field obtained by adjoining an element $\omega$ to $\mathbb{Q}$. Then the graph $\Gamma_C(M_3(\F))$ is disconnected. Consequently, $\Gamma_C(P_8(\F))$ is also disconnected.
\end{corollary}

\begin{proof}
Let $\zeta \in \mathbb{C}$ be a primitive ninth root of unity. Then
\[
    [ \mathbb{Q}[\zeta] : \mathbb{Q} ] = \varphi(9) = 6,
\]
where $\varphi$ is Euler's totient function. The minimal polynomial of $\zeta$ over $\mathbb{Q}$ is the ninth cyclotomic polynomial
\[
    \Phi_9(x) = x^6 + x^3 + 1.
\]
Note that $\zeta^3$ is a primitive cube root of unity.  Without loss of generality we may assume that $\omega = \zeta^3$.  The minimal polynomial of $\omega$ over $\mathbb{Q}$ is $x^2 + x + 1$, so
\[
    [ \mathbb{Q}[\omega] : \mathbb{Q} ] = 2.
\]
Since $\omega \in \mathbb{Q}[\zeta]$, the field $\mathbb{Q}[\omega]$ is a subfield of $\mathbb{Q}[\zeta]$.  Then, by the tower law,
\[
    [\mathbb{Q}[\zeta] : \mathbb{Q}[\omega]] = \frac{[\mathbb{Q}[\zeta] : \mathbb{Q}]}{[\mathbb{Q}[\omega] : \mathbb{Q}]} = \frac{6}{2} = 3.
\]
Thus $\mathbb{Q}[\zeta]$ is an extension of $\mathbb{Q}[\omega]$ of degree $3$.

Now assume that the extension $\mathbb{Q}[\zeta]/\mathbb{Q}[\omega]$ contains an intermediate field $E$.  Then, again by the tower law,
\[
    [\mathbb{Q}[\zeta] : E] \cdot [E : \mathbb{Q}[\omega]] = [\mathbb{Q}[\zeta] : \mathbb{Q}[\omega]] = 3.
\]
Since $3$ is a prime number, the field $E$ must coincide either with $\mathbb{Q}[\zeta]$ or with $\mathbb{Q}[\omega]$.  Therefore, there are no proper intermediate fields. Applying Theorem~\ref{theorem:connectivity-criterion-of-graph-Mn} and Proposition~\ref{proposition:commuting-graph-isomorphism}, we obtain the required result.
\end{proof}

\section{Path of length~$2$ between idempotents} \label{section:path-idempotents}

In this section we show that, in any Okubo algebra with idempotents over an arbitrary field~$\F$, the centralizers of any two idempotents always have a nontrivial intersection.  We will use this observation in the next section for studying the commuting graph of the real pseudo-octonions.

\begin{lemma} \label{lemma:length-two-path-idempotents}
Let $e,f \in \Okubo$ be nonzero idempotents.  Consider the element 
\[
    x = (e+f)*(e+f) - 2(e+f) = e*f + f*e - e - f.
\]
If $x \neq 0$, then $[e] \leftrightarrow [x] \leftrightarrow [f]$ in $\Gamma_C(\Okubo)$.
\end{lemma}

\begin{proof}
By Proposition~\ref{proposition:norm-of-idempotent}, $\n(e) = \n(f) = 1$. Using Theorem~\ref{theorem:symmetric-flexible}(2) and Eq.~\eqref{equation:lin-of-symm}, one can verify that
\begin{align*}
    e*x &= e*(e*f) + e*(f*e) - e*e - e*f \\
    &= \n(e,f) e - f*(e*e) + \n(e) f - e - e*f \\
    &= (\n(e,f) - 1)e + f - e*f - f*e,\\
    x*e &= (e*f)*e + (f*e)*e - e*e - f*e \\
    &= \n(e) f + \n(e,f) e - (e*e)*f - e - f*e \\
    &= (\n(e,f) - 1)e + f - e*f - f*e,
\end{align*}
so $e*x = x*e$.  Similarly, $f*x = x*f$.
\end{proof}

\begin{lemma} \label{lemma:common-idempotents}
If $x = 0$ in Lemma~\ref{lemma:length-two-path-idempotents}, then $f = e + z$, where
\begin{enumerate}[{\rm (1)}]
    \item $\n(z) = \n(e,z) = 0$,
    \item $z*z = 0$,
    \item $z = e*z + z*e$.
\end{enumerate}
\end{lemma}

\begin{proof}
We have $0 = e*x = (\n(e,f) - 2)e - x = (\n(e,f) - 2)e$, so $\n(e,f) = 2$.  Let $z = f - e$.  Then 
\begin{align*}
    \n(z) &= \n(f - e) = \n(f) + \n(e) - \n(e,f) = 1 + 1 - 2 = 0,\\
    \n(z,e) &= \n(e+z) - \n(e) - \n(z) = 1 - 1 - 0 = 0.
\end{align*}
Substituting $f = e + z$ into the equalities $x = 0$ and $f*f = f$, we obtain
\begin{align*}
    0 &= x = e*(e+z) + (e+z)*e - e - (e+z) = e*z + z*e - z,\\
    0 &= f*f - f = (e+z)*(e+z) - (e+z) = (e*z + z*e - z) + z*z = z*z. \qedhere
\end{align*}
\end{proof}

\begin{remark} \label{remark:anisotropic}
In particular, if the Okubo algebra $\Okubo$ has no zero divisors, that is, the norm on $\Okubo$ is anisotropic (see \cite[Lemma~2.1]{EM91}), then the condition $x = 0$ in Lemma~\ref{lemma:length-two-path-idempotents} is possible only if $e = f$.
\end{remark}
	
\begin{definition}
Let $\A$ be an arbitrary algebra. The {\em centralizer} of an element~$a \in \A$ is the set of elements in $\A$ that commute with it:
\[
    C_\A(a) = \big\{ b \in \A \: | \: ab=ba \big\}.
\]
\end{definition}

\begin{lemma} \label{lemma:idempotents-algebraically-closed}
Let the field~$\F$ be algebraically closed, $\operatorname{char}\F \neq 3$, and let $\Okubo = P_8(\F)$ be the pseudo-octonion algebra over~$\F$. Then the elements $e, z \in \Okubo$ which satisfy the conditions of Lemma~\ref{lemma:common-idempotents}, with $z$ being nonzero, have the form
\[
    z = \begin{pmatrix}
        0 &  0 & 1 \\
        0 & 0 & 0 \\
        0 & 0 & 0
    \end{pmatrix}, \quad
    e \in \left\{ \begin{pmatrix}
        -1 &  0 & 0 \\
        0 & -1 & 0 \\
        0 & 0 & 2
    \end{pmatrix}, \; \begin{pmatrix}
        2 &  0 & 0 \\
        0 & -1 & 0 \\
        0 & 0 & -1
    \end{pmatrix} \right\}
\]
up to conjugation by matrices from $GL_3(\F)$. Moreover, for
\[
    g = \begin{pmatrix}
        -1 &  0 & 0 \\
        0 & 2 & 0 \\
        0 & 0 & -1
    \end{pmatrix}
\]
we have $g \in C_{\Okubo}(e) \cap C_{\Okubo}(z)$, i.e.\ $[e] \leftrightarrow [g] \leftrightarrow [f]$.
\end{lemma}

\begin{proof}
Since $\n(z) = 0$, it can be easily seen that $0 = z*z = zz$, and then the desired form of the matrix~$z$ follows immediately from its Jordan normal form.  Next, Eq.~\eqref{equation:product} implies that $z = e*z + z*e = ez + ze - 2/3 \cdot \tr(ez) I$. A straightforward computation shows that
\[
    e = 
    \begin{pmatrix}
        \alpha &  \beta & \gamma \\
        0 & -1 & \delta \\
        0 & 0 & 1-\alpha
    \end{pmatrix}
\]
for some $\alpha,\beta,\gamma,\delta \in \F$.  Since $e$ is idempotent, by~\cite[p.~10]{Eld15}, it is conjugate to the matrix
\[
    \begin{pmatrix}
        -1 &  0 & 0 \\
        0 & -1 & 0 \\
        0 & 0 & 2
    \end{pmatrix}.
\]
Hence either $\alpha = -1$ or $\alpha = 2$.  Writing out explicitly $e = e*e = ee - \tr(ee)/3 \cdot I$, we obtain that
\[
    e = \begin{pmatrix}
        -1 &  0 & \gamma \\
        0 & -1 & \delta \\
        0 & 0 & 2
    \end{pmatrix}
    \quad \text{or} \quad
    e = \begin{pmatrix}
        2 &  \beta & \gamma \\
        0 & -1 & 0 \\
        0 & 0 & -1
    \end{pmatrix}.
\]
Consider
\[
    C = \begin{pmatrix}
        1 &  0 & -\gamma/3 \\
        0 & 1 & -\delta/3 \\
        0 & 0 & 1
    \end{pmatrix}
    \quad \text{or} \quad
    C = \begin{pmatrix}
        1 & -\beta/3 & -\gamma/3 \\
        0 & 1 & 0 \\
        0 & 0 & 1
    \end{pmatrix},
\]
respectively.  Then $CzC^{-1} = z$, and
\[
    CeC^{-1} = \begin{pmatrix}
        -1 &  0 & 0 \\
        0 & -1 & 0 \\
        0 & 0 & 2
    \end{pmatrix}
    \quad \text{or} \quad
    CeC^{-1} = \begin{pmatrix}
        2 & 0 & 0 \\
        0 & -1 & 0 \\
        0 & 0 & -1
    \end{pmatrix}.
\]
The condition $g \in C_{\Okubo}(e) \cap C_{\Okubo}(z)$ can be  verified directly by using Remark~\ref{remark:commutativity}.
\end{proof}

\begin{corollary} \label{corollary:idempotents-distance-two}
Let $\Okubo$ be an Okubo algebra over an arbitrary field~$\F$, and let $e,f \in \Okubo$ be nonzero idempotents. Then $\D([e],[f]) \le 2$ in $\Gamma_C(\Okubo)$.
\end{corollary}

\begin{proof}
If $x \neq 0$ in Lemma~\ref{lemma:length-two-path-idempotents}, then $[e] \leftrightarrow [x] \leftrightarrow [f]$ in $\Gamma_C(\Okubo)$.  Otherwise, the conditions of Lemma~\ref{lemma:common-idempotents} hold.  If $z = 0$, then $e = f$, and $\D([e],[f]) = 0$. Thus we may assume that $z \neq 0$.

First, we consider the case when $\operatorname{char}\F \neq 3$.  Let $L_e$ and $R_e$ be the linear operators of left and right multiplication by $e$, respectively. Then $C_{\Okubo}(e) = \ker(L_e - R_e)$.  Similarly, $C_{\Okubo}(z) = \ker(L_z - R_z)$.  Passing to the algebraic closure $\overline{\F}$ of~$\F$, we obtain the Okubo algebra $\Okubo_{\overline{\F}} = \Okubo \otimes_{\F} \overline{\F}$ over $\overline{\F}$, where the elements $e$ and $z$ still satisfy the conditions of Lemma~\ref{lemma:common-idempotents}.  The matrices of the linear operators $L_e, R_e, L_z, R_z$ written in some basis of~$\Okubo$ are also the matrices of the corresponding operators on $\Okubo_{\overline{\F}}$, in the same basis.  By Lemma~\ref{lemma:idempotents-algebraically-closed}, there exists a nonzero element $g \in C_{\Okubo_{\overline{\F}}}(e) \cap C_{\Okubo_{\overline{\F}}}(z) = \ker_{\overline{\F}}(L_e - R_e) \cap \ker_{\overline{\F}}(L_z - R_z)$.  Since the existence of a nonzero solution for a system of linear equations does not depend on passing to the algebraic closure, there exists a nonzero element $g' \in \ker_{\F}(L_e - R_e) \cap \ker_{\F}(L_z - R_z) = C_{\Okubo}(e) \cap C_{\Okubo}(z)$, and then $[e] \leftrightarrow [g'] \leftrightarrow [f]$.

Consider now the case when $\operatorname{char}\F = 3$.  Let $y = z + e*z$.  If $y = 0$, then $z*e = z - e*z = 2z = -z = e*z$, hence $[e] \leftrightarrow [z] \leftrightarrow [f]$.

Now assume that $y \neq 0$.  Using Theorem~\ref{theorem:symmetric-flexible}(2) and Eq.~\eqref{equation:lin-of-symm}, we obtain
\begin{align*}
    z*y &= z*(z + e*z) = z*z + \n(z)e = 0,\\
    y*z &= (z + e*z)*z = z*z + \n(e,z)z - (z*z)*e = 0,
\end{align*}
so $y \in C_{\Okubo}(z)$.  We show that $y \in C_{\Okubo}(e)$. Indeed,
\begin{align*}
    e*y &= e*z + e*(e*z) = e*z + \n(e,z)e - z*(e*e) = e*z - z*e,\\
    y*e &= z*e + (e*z)*e = z*e + z = e*z + 2z*e.
\end{align*}
Since $\operatorname{char}\F = 3$, we have $e*y = y*e$, so $[e] \leftrightarrow [y] \leftrightarrow [f]$.
\end{proof}
	
\section{Commuting graph of the real pseudo-octonions}

Finally, we consider the case of the real Okubo algebra $\widetilde{P}_8(\R)$. In this section we show that the commuting graph $\Gamma_C(\widetilde{P}_8(\R))$ is connected, and its diameter equals four, as in the case of an Okubo algebra over an algebraically closed field.

Recall that the Gell--Mann matrices are a set of eight traceless Hermitian matrices which satisfy the condition $\tr(\lambda_j\lambda_k) = 2 \delta_{jk}$, see~\cite[p.~1074]{Gell-Mann}:
\begin{align*}
    \lambda_1 &=
    \begin{pmatrix}
        0 & 1 & 0 \\
        1 & 0 & 0 \\
        0 & 0 & 0 \\
    \end{pmatrix},
    &
    \lambda_2 &= 
    \begin{pmatrix}
        0 & -i & 0 \\
        i & 0 & 0 \\
        0 & 0 & 0 \\
    \end{pmatrix},\\
    \lambda_3 &= 
    \begin{pmatrix}
        1 & 0 & 0 \\
        0 & -1 & 0 \\
        0 & 0 & 0 \\
    \end{pmatrix},
    &
    \lambda_4 &=
    \begin{pmatrix}
        0 & 0 & 1 \\
        0 & 0 & 0 \\
        1 & 0 & 0 \\
    \end{pmatrix},\\
    \lambda_5 &= 
    \begin{pmatrix}
        0 & 0 & -i \\
        0 & 0 & 0 \\
        i & 0 & 0 \\
    \end{pmatrix},
    &
    \lambda_6 &= 
    \begin{pmatrix}
        0 & 0 & 0 \\
        0 & 0 & 1 \\
        0 & 1 & 0 \\
    \end{pmatrix},\\
    \lambda_7 &=
    \begin{pmatrix}
        0 & 0 & 0 \\
        0 & 0 & -i \\
        0 & i & 0 \\
    \end{pmatrix},
    &
    \lambda_8 &= 
    \begin{pmatrix}
        \frac{1}{\sqrt{3}} & 0 & 0 \\
        0 & \frac{1}{\sqrt{3}} & 0 \\
        0 & 0 & -\frac{2}{\sqrt{3}} \\
    \end{pmatrix}.
\end{align*}

As noted in~\cite[p.~45]{Oku95}, the elements $e_j = \sqrt{3} \lambda_j \: (j=1,\dots,8)$ form an orthonormal basis of the algebra~$\widetilde{P}_8(\R)$ with respect to the inner product $\n(\cdot,\cdot)/2$.

\begin{definition}[{\cite[p.~18]{Oku78}}]
The subalgebra $\widetilde{P}_4(\R) = \spn \{ e_1, e_2, e_3, e_8\}$ is called the {\em real pseudo-quaternion algebra}.
\end{definition}

Define an element $f \in \widetilde{P}_8(\R)$ by
\begin{equation} \label{equation:idempotent}
    f = \begin{pmatrix}
        -1 &  0 & 0 \\
        0  & -1 & 0 \\
        0  &  0 & 2
    \end{pmatrix}.    
\end{equation}
According to~\cite[p.~10]{Eld15}, the element $f$ is the unique nonzero idempotent in $P_8(\C)$ up to conjugation. Since $\widetilde{P}_8(\R)$ can be considered as a subset of $P_8(\C)$, every idempotent element of $\widetilde{P}_8(\R)$ is also an idempotent element in $P_8(\C)$. It is well known that any Hermitian matrix is unitarily similar to its Jordan normal form, and thus we immediately obtain the following proposition.
    
\begin{proposition} \label{proposition:uniqe-idempotent}
The element~\eqref{equation:idempotent} is the unique idempotent in $\widetilde{P}_8(\R)$ up to unitary conjugation.
\end{proposition}

We determine the explicit form of the centralizer of the idempotent $f$.
    
\begin{proposition} \label{proposition:centre-of-p4}
    $C_{\widetilde{P}_8(\R)} (f) = \widetilde{P}_4 (\R)$.
\end{proposition}

\begin{proof}
Notice that the subalgebra $\widetilde{P}_4(\R)$ consists of block diagonal matrices with a $2\times 2$ upper block and a $1\times 1$ lower block. In particular, $f \in \widetilde{P}_4(\R)$, and since the corresponding blocks of the matrix $f$ are scalar matrices, we have $\widetilde{P}_4(\R) \subseteq C_{\widetilde{P}_8(\R)}(f)$.

On the other hand, the centralizer of a nonzero idempotent in an arbitrary Okubo algebra over a field~$\F$, $\chrs \F \neq 3$, is a four-dimensional subalgebra~\cite[Corollary 5.8]{Eld18}. Therefore, $C_{\widetilde{P}_8(\R)}(f) = \widetilde{P}_4(\R)$.
\end{proof}

\begin{remark} \label{remark:centralizer-of-diag-idempotent}
Since $f$ is the unique idempotent in $\widetilde{P}_8(\R)$ up to unitary conjugation, there are exactly three diagonal idempotents in $\widetilde{P}_8(\R)$, namely, $f$, $ufu^*$, and $vfv^*$, where
\begin{equation*}
u =
\begin{pmatrix}
    0 & 0 & 1 \\
    0 & 1 & 0 \\
    1 & 0 & 0
\end{pmatrix}, \quad
v =
\begin{pmatrix}
    1 & 0 & 0 \\
    0 & 0 & 1 \\
    0 & 1 & 0
\end{pmatrix}. 
\end{equation*}
It follows from Proposition~\ref{proposition:centre-of-p4} that their centralizers are $\widetilde{P}_4 = \widetilde{P}_4(\R)$, $L_4 = u\widetilde{P}_4(\R)u^*$, and $M_4 = v\widetilde{P}_4(\R)v^*$, respectively.
\end{remark}

In what follows it will be convenient to use the definition of a nonderogatory matrix, whose centralizer is well known.
	
\begin{definition}
A matrix $A \in M_n(\F)$ is called {\em nonderogatory} if its minimal and characteristic polynomials coincide.

In particular, if the characteristic polynomial splits into linear factors over~$\F$, then $A$ is nonderogatory if and only if each of its eigenvalues has geometric multiplicity~$1$.
\end{definition}
	
\begin{theorem}[{\cite[Theorem 3.2.4.2, Problem 3.2.P2]{Horn}}]
    \label{theorem:commutation-of-nonderogatory-matrix}
A matrix $a \in M_n(\F)$ is nonderogatory if and only if $C_{M_n(\F)}(a)$ coincides with the unital subalgebra generated by $a$.
\end{theorem}

Some other equivalent conditions which characterize nonderogatory matrices over fields with at least $n$ elements are given in~\cite[Theorem 2.8]{DGKO13}. An analogue of Theorem~\ref{theorem:commutation-of-nonderogatory-matrix} can also be proved for nonderogatory matrices in the pseudo-octonion algebra.
	
\begin{proposition} \label{proposition:centralizer-of-nonderogatory-matrix}
Let $\chrs \F \neq 3$ and $\omega \in \F$. Consider a nonzero element $a \in P_8(\F)$. If the matrix $a$ is nonderogatory, then $C_{P_8(\F)} (a) = \spn \big\{ a, a*a \big\} = \alg{a}$, and the dimension of the centralizer equals~$2$.
\end{proposition}

\begin{proof}
By Theorem~\ref{theorem:commutation-of-nonderogatory-matrix}, we have
\(
    C_{M_3(\F)} (a) = \spn \{ I, a, aa \},
\)
and its dimension equals $3$. Then Remark~\ref{remark:commutativity} implies that
\[
    C_{P_8(\F)} (a) = \{ b \in C_{M_3 (\F)}(a) \; | \; b \in P_8(\F) \} = \{ b \in C_{M_3 (\F)}(a) \; | \; \tr(b) = 0 \}.
\]
Since $a*a = aa - \tr(aa)/3 \cdot I$ and $\tr(a) = \tr(a \ast a) = 0$, the desired statement follows immediately.
\end{proof}

\begin{remark} \label{remark:centralizer-of-nonderogatory-matrix}
Similarly, one can show that if a matrix $a \in \widetilde{P}_8(\R)$ is nonderogatory, then $C_{\widetilde{P}_8(\R)} (a) = \spn_{\R} \{ a, a*a \} = \alg{a}$, since \[C_{\widetilde{P}_8(\R)}(a) =  \{ b \in C_{P_8(\C)}(a) \; | \; b^* = b \}. \]
\end{remark}

The next proposition shows that the set of derogatory matrices in $\widetilde{P}_8(\R)$ coincides with the union of the lines passing through nonzero idempotents.

\begin{proposition} \label{proposition:nonderogatory-matrix-criterion}
Let $x \in \widetilde{P}_8(\R)$, $x \neq 0$. The set $ \F x$ contains a nonzero idempotent if and only if $x$ is a derogatory matrix.
\end{proposition}

\begin{proof}
If the set $\F x$ contains a nonzero idempotent $g$, then, by Proposition~\ref{proposition:uniqe-idempotent}, the element $x$ has an eigenvalue with geometric multiplicity~$2$. Hence the matrix $x$ is derogatory.

Conversely, if $x$ is a derogatory matrix, then $x$ has an eigenvalue with geometric multiplicity~$2$. Since $x$ is Hermitian and $\tr(x) = 0$, its eigenvalues are $(-\lambda,-\lambda,2\lambda)$ for some $0 \neq \lambda \in \R$. It follows that~$x/\lambda$ is an idempotent.
\end{proof}

As shown above, idempotents have centralizers of the largest dimension in $\widetilde{P}_8(\R)$. We now prove that any vertex of the commuting graph is adjacent to an idempotent.
    
\begin{lemma} \label{lemma:path-to-idempotent}
For any nonzero element $x \in \widetilde{P}_8(\R)$ there exists a nonzero idempotent $g \in \widetilde{P}_8(\R)$ such that $[x] \leftrightarrow [g]$ in $\Gamma_C(\widetilde{P}_8(\R))$.
\end{lemma}

\begin{proof}
Since the matrix $x$ is Hermitian, there exists a unitary matrix $u \in M_3(\C)$ such that $x = udu^*$, where $d$ is a diagonal matrix. Let $g = ufu^*$. Then $g$ is an idempotent, and $[d] \leftrightarrow [f]$ implies $[x] \leftrightarrow [g]$.
\end{proof}

Combining Lemma~\ref{lemma:path-to-idempotent} with the results of the previous section, which show that the distance between any two idempotents is at most two, we obtain the main statement on the commuting graph of the real pseudo-octonions.
		
\begin{theorem} \label{theorem:real-commutativity-graph}
The graph $\Gamma_C(\widetilde{P}_8(\R))$ is connected, and its diameter equals~$4$.
\end{theorem}

\begin{proof}
It follows from Lemmas~\ref{lemma:length-two-path-idempotents} and~\ref{lemma:path-to-idempotent} and Remark~\ref{remark:anisotropic} that $\diam(\Gamma_C(\widetilde{P}_8(\R))) \leq 4$.
On the other hand, consider the matrices
\begin{equation*}
    x = 
    \begin{pmatrix}
        0 & i & -i \\
        -i & 0 & i \\
        i & -i & 0
    \end{pmatrix}, \quad
    y = - x \ast x = 
    \begin{pmatrix}
        0 & 1 & 1 \\
        1 & 0 & 1 \\
        1 & 1 & 0
    \end{pmatrix}.
\end{equation*}
Then the elements $x$ and $y = -x*x$ are linearly independent. Consider the two-dimensional subalgebra $\alg{x} = \spn \{ x, y \}$ generated by the element $x$. It can be easily seen that
\[
\alg{x} \cap (\widetilde{P}_4 \cup L_4 \cup M_4) = \{ 0 \},
\]
where the subalgebras $L_4$ and $M_4$ are defined in Remark~\ref{remark:centralizer-of-diag-idempotent}.

By Proposition~\ref{proposition:quadratic-subalgebra}, if nonzero elements $a$ and $a'$ are not proportional to idempotents, then the subalgebras $\alg{a}$ and $\alg{a'}$ are two-dimensional. Then $a' \in \alg{a}$ implies $\alg{a'} = \alg{a}$. Furthermore, by Remark~\ref{remark:centralizer-of-nonderogatory-matrix} and Proposition~\ref{proposition:nonderogatory-matrix-criterion}, we have $C_{\widetilde{P}_8(\R)} (a) = \alg{a}$. Thus any path from $[a]$ to $[b]$, where $b \in \widetilde{P}_8(\R) \setminus \alg{a}$, must contain a vertex of the form $[g]$, where $g \in \alg{a}$ is an idempotent.

Therefore, a path from $[x]$ to $[e_3]$ must contain vertices $[g]$ and $[g']$, where $g \in \alg{x}$ and $g' \in \alg{e_3}$ are idempotents. 
Since $g'$ is a diagonal idempotent, it follows from Remark~\ref{remark:centralizer-of-diag-idempotent} that $C_{\widetilde{P}_8(\R)} (g') \subseteq \widetilde{P}_4 \cup M_4 \cup L_4$, and hence $g \centernot\leftrightarrow g'$, so $\D([x], [e_3]) \geq 4$. Consequently, $\D([x], [e_3]) = 4$ and $\diam(\Gamma_C(\widetilde{P}_8(\R))) = 4$.
\end{proof}

\begin{remark}
It follows from the proof of Theorem~\ref{theorem:real-commutativity-graph} that the maximum distance can be reached only between those vertices of $\Gamma_C(\widetilde{P}_8(\R))$ which are not proportional to idempotents. In other words, if $\D([x], [y]) = 4$ in $\Gamma_C(\widetilde{P}_8(\R))$, then $\F x$ and $\F y$ do not contain nonzero idempotents, i.e., the matrices $x$ and $y$ are nonderogatory, see Proposition~\ref{proposition:nonderogatory-matrix-criterion}. This statement is similar to Theorem~1.1 of~\cite{DKO12}, although being its weaker version. According to the above-mentioned theorem, if $n \geq 3$ and the field~$\F$ is algebraically closed, then a matrix $A \in M_n(\F)$ is  nonderogatory if and only if there exists a matrix $X \in M_n(\F)$ such that $\D(A,X) = 4$ in $\Gamma_C(M_n(\F))$. 
\end{remark}

\medskip

The authors are grateful to Professor Alexander E. Guterman for posing the problem, careful attention to the work, and fruitful discussions.


\begin{thebibliography}{99}            
            \bibitem{Akbari2}            
            S. Akbari, H. Bidkhori, A. Mohammadian, \textit{Commuting graphs of matrix algebras.} --- Comm. Algebra \textbf{36}, No. 11 (2008), 4020--4031. DOI: 10.1080/00927870802174538

			\bibitem{Akbari}
			S. Akbari, A. Mohammadian, H. Radjavi, P. Raja, \emph{On the diameters of commuting graphs.} --- Linear Algebra Appl. \textbf{418}, No. 1 (2006), 161--176. DOI: 10.1016/j.laa.2006.01.029

            \bibitem{Akbari3}
            S. Akbari, P. Raja, \emph{Commuting graphs of some subsets in simple rings.} --- Linear Algebra Appl. \textbf{416}, Nos. 2–3 (2006), 1038--1047. DOI: 10.1016/j.laa.2006.01.006
			
            \bibitem{DGKO13}
            G. Dolinar, A. Guterman, B. Kuzma, P. Oblak, \emph{Extremal matrix centralizers.} --- Linear Algebra Appl. \textbf{438}, No. 7 (2013), 2904--2910. DOI: 10.1016/j.laa.2012.12.010
			
			\bibitem{DKO12}
            G. Dolinar, B. Kuzma, P. Oblak, \emph{On maximal distances in a commuting graph.} --- 	Electron. J. Linear Algebra \textbf{23}, No. 1 (2012), 243--256. DOI:10.13001/1081-3810.1518 
            
            \bibitem{Eld97}
			A. Elduque, {\em Symmetric composition algebras.} --- J. Algebra \textbf{196}, No. 1 (1997), 282--300. DOI: 10.1006/jabr.1997.7071
			
			\bibitem{Eld99}
			A. Elduque, {\em Okubo algebras and twisted polynomials.} --- 
            In: Recent Progress in Algebra (Taejon/Seoul, 1997), 
            Contemp. Math. \textbf{224} (1999), 101--109. 
            Providence, RI: Amer. Math. Soc.
            DOI: 10.1090/conm/224
			
			\bibitem{Eld15}
			A. Elduque, {\em Okubo algebras: automorphisms, derivations and idempotents.} ---
            In: Lie algebras and related topics, 
            Contemp. Math. \textbf{652} (2015), 61--73. 
            Providence, RI: Amer. Math. Soc.
            DOI: 10.1090/conm/652/12953

            \bibitem{Eld18}
            A. Elduque, {\em Order $3$ elements in $G_2$ and idempotents in symmetric composition algebras.} --- Canad. J. Math. \textbf{70}, No. 5 (2018), 1038--1075. DOI: 10.4153/CJM-2017-039-6
            
                \bibitem{Eld23}
            A. Elduque, {\em Okubo algebras with isotropic norm.} --- Springer Proc. Math. Stat. \textbf{427} (2023), 287--301. 
            DOI: 10.1007/978-3-031-32707-0\_18
			
			\bibitem{EM91}
			A. Elduque, H.~Ch. Myung, {\em Flexible composition algebras and Okubo algebras.} --- Comm. Algebra \textbf{19}, No. 4 (1991), 1197--1227. DOI: 10.1080/00927879108824198
			
			\bibitem{EM93}
			A. Elduque, H.~Ch. Myung, {\em On flexible composition algebras.} --- Comm. Algebra \textbf{21}, No. 7 (1993), 2481--2505. DOI: 10.1080/00927879308824688
			
			\bibitem{EP96}
			A. Elduque, J.~M. P\'erez, {\em Composition algebras with associative bilinear form.} --- Comm. Algebra \textbf{24}, No. 3 (1996), 1091--1116. DOI: 10.1080/00927879608825625
			
			\bibitem{Fau88}
			J.~R. Faulkner, {\em Finding octonion algebras in associative algebras.} --- Proc. Amer. Math. Soc. \textbf{104}, No. 4 (1988), 1027--1030. DOI: 10.2307/2047585 
			
			\bibitem{Gell-Mann}
			M. Gell-Mann, {\em Symmetries of baryons and mesons. } --- Phys. Rev. \textbf{125}, No. 3 (1962), 1067–1084. DOI: 10.1103/physrev.125.1067
			
			\bibitem{Horn}
            R. A. Horn, C. R. Johnson, \emph{Matrix Analysis}, Second ed. --- Cambridge, UK: Cambridge University Press (2012). DOI: 10.1017/CBO9781139020411
			
			\bibitem{KMRT98}
			M.-A. Knus, A. Merkurjev, M. Rost, J.-P. Tignol, \emph{The Book of Involutions.} --- Amer. Math. Soc. Colloq. Publ. \textbf{44} (1998). Providence, RI: Amer. Math. Soc. DOI: 10.1090/coll/044

            \bibitem{Morrison}
            R. Morrison, \emph{Commuting graphs of $p$-adic matrices.} --- arXiv:2407.13848 (2024).
			
			\bibitem{Oku78}
			S. Okubo, {\em Pseudo-quaternion and pseudo-octonion algebras.} --- Hadronic J. \textbf{1}, No. 4 (1978), 1250--1278.
			
			\bibitem{Oku78D}
            S. Okubo, {\em Deformation of the Lie-admissible pseudo-octonion algebra into the octonion algebra.} --- Hadronic J. \textbf{1}, No. 5 (1978), 1383--1431
			
			\bibitem{Oku95}
			S. Okubo, \emph{Introduction to octonion and other non-associative algebras in physics.} --- Cambridge, UK: Cambridge University Press (1995). DOI: 10.1017/CBO9780511524479
			
			\bibitem{OO81b}
			S. Okubo, J.~M. Osborn, {\em Algebras with nondegenerate associative symmetric bilinear forms permitting composition. II.} --- Comm. Algebra. \textbf{9}, No. 20 (1981), 2015--2073. DOI: 10.1080/00927878108822695
			
			\bibitem{Pet69}
			H.-P. Petersson, {\em Eine Identit\"at f\"unften Grades, der gewisse Isotope von Kompositions-Algebren gen\"ugen.} --- Math. Z. \textbf{109} (1969), 217–238.

            \bibitem{Shitov16}
            Ya. N. Shitov, \emph{A matrix ring with commuting graph of maximal diameter.} --- J.~Combin. Theory, Ser.~A \textbf{141} (2016), 127--135. DOI: 10.1016/j.jcta.2016.02.008.
            
            \bibitem{Shitov}
            Ya. N. Shitov, {\em Distances on the commuting graph of the ring of real matrices.} --- Matem. Zametki \textbf{103}, No. 5 (2018), 765–768; English transl., Math. Notes \textbf{103}(5) (2018), 832–835. DOI: 10.1134/S0001434618050152

		\end{thebibliography}
	\end{document}